\documentclass[11pt]{article}
\usepackage{amsmath,amssymb,amsthm}

\newtheorem{theorem}{Theorem}[section]
\newtheorem{lemma}[theorem]{Lemma}
\newtheorem{proposition}[theorem]{Proposition}
\newtheorem{corollary}[theorem]{Corollary}

\theoremstyle{definition}
\newtheorem{definition}[theorem]{Definition}

\numberwithin{equation}{section}

\renewcommand{\phi}{\varphi}

\newcommand{\BF}{\operatorname{BF}}
\newcommand{\Hom}{\operatorname{Hom}}
\newcommand{\Homeo}{\operatorname{Homeo}}
\newcommand{\id}{\operatorname{id}}
\newcommand{\Ker}{\operatorname{Ker}}
\newcommand{\orb}{\operatorname{orb}}

\newcommand{\K}{\mathbb{K}}

\newcommand{\N}{\mathbb{N}}
\newcommand{\Z}{\mathbb{Z}}

\title{Continuous orbit equivalence of topological Markov shifts 
and Cuntz-Krieger algebras}
\author{Kengo Matsumoto \\
Department of Mathematics \\
Joetsu University of Education \\
Joetsu, Niigata 943-8512, Japan
\and
Hiroki Matui \\
Graduate School of Science \\
Chiba University \\
Inage-ku, Chiba 263-8522, Japan}
\date{}

\begin{document}
\maketitle

\begin{abstract}
Let $A,B$ be square irreducible matrices with entries in $\{0,1\}$. 
We will show that if the one-sided topological Markov shifts 
$(X_A,\sigma_A)$ and $(X_B,\sigma_B)$ are continuously orbit equivalent, 
then the two-sided topological Markov shifts 
$(\bar X_A,\bar\sigma_A)$ and $(\bar X_B,\bar\sigma_B)$ are 
flow equivalent, and hence $\det(\id-A)=\det(\id-B)$. 
As a result, the one-sided topological Markov shifts 
$(X_A,\sigma_A)$ and $(X_B,\sigma_B)$ are continuously orbit equivalent 
if and only if the Cuntz-Krieger algebras 
$\mathcal{O}_A$ and $\mathcal{O}_B$ are isomorphic and 
$\det(\id-A)=\det(\id-B)$. 
\end{abstract}

\section{Introduction}

The interplay between orbit equivalence of topological dynamical systems 
and the theory of $C^*$-algebras has been studied by many authors. 
T. Giordano, I. F. Putnam and C. F. Skau \cite{GPS95crelle} have proved that 
two minimal homeomorphisms on a Cantor set are strongly orbit equivalent 
if and only if the associated $C^*$-crossed products are isomorphic. 
M. Boyle and J. Tomiyama \cite{T96Pacific,BT98JMSJ} have studied 
relationships between orbit equivalence and $C^*$-crossed products 
for topologically free homeomorphisms on compact Hausdorff spaces. 

In this paper, we classify one-sided irreducible topological Markov shifts 
up to continuous orbit equivalence 
and show that there exists a close connection with the Cuntz-Krieger algebras. 
The class of one-sided topological Markov shifts is 
an important class of topological dynamical systems on Cantor sets, 
though they are not homeomorphisms but local homeomorphisms. 
In \cite{Matsumoto10PJM} the first-named author introduced the notion 
of continuous orbit equivalence for one-sided topological Markov shifts 
(see Definition \ref{defofcoe}) 
and proved that one-sided topological Markov shifts 
$(X_A,\sigma_A)$ and $(X_B,\sigma_B)$ for irreducible matrices $A$ and $B$ 
with entries in $\{0,1\}$ are continuously orbit equivalent 
if and only if there exists an isomorphism 
between the Cuntz-Krieger algebras $\mathcal{O}_A$ and $\mathcal{O}_B$ 
preserving their canonical Cartan subalgebras 
$\mathcal{D}_A$ and $\mathcal{D}_B$. 
The second-named author \cite{Matui12PLMS,Matui12} studied 
the associated \'etale groupoids $G_A$ and 
their homology groups $H_n(G_A)$ and topological full groups $[[G_A]]$. 
In fact, the two shifts are continuously orbit equivalent if and only if 
$G_A$ is isomorphic to $G_B$ (see Theorem \ref{coe=cartaniso}). 
In \cite{Matsumoto13PAMS} it was also shown that 
if $\mathcal{O}_A$ is isomorphic to $\mathcal{O}_B$ and 
$\det(\id-A)=\det(\id-B)$, 
then there exists an isomorphism $\Psi:\mathcal{O}_A\to\mathcal{O}_B$ 
such that $\Psi(\mathcal{D}_A)=\mathcal{D}_B$, 
and hence the one-sided topological Markov shifts 
$(X_A,\sigma_A)$ and $(X_B,\sigma_B)$ are continuously orbit equivalent. 
Since there were no known examples of irreducible matrices $A,B$ such that 
$(X_A,\sigma_A)$ and $(X_B,\sigma_B)$ are continuously orbit equivalent and 
$\det(\id-A)\neq\det(\id-B)$, 
the first-named author \cite[Section 6]{Matsumoto13PAMS} presented 
the following conjecture: 
the determinant $\det(1-A)$ is an invariant 
for the continuous orbit equivalence class of $(X_A,\sigma_A)$. 
In the present article we confirm this conjecture. 
In other words we show that 
$(X_A,\sigma_A)$ and $(X_B,\sigma_B)$ are continuously orbit equivalent 
if and only if $\mathcal{O}_A$ is isomorphic to $\mathcal{O}_B$ and 
$\det(\id-A)=\det(\id-B)$ (Theorem \ref{classification}). 

Our proof is closely related to another notion of equivalence for shifts, 
namely flow equivalence for two-sided topological Markov shifts. 
Two-sided topological Markov shifts 
$(\bar X_A,\bar\sigma_A)$ and $(\bar X_B,\bar\sigma_B)$ are 
said to be flow equivalent 
if there exists an orientation-preserving homeomorphism 
between their suspension spaces (\cite{PS75Top}). 
Two characterizations of the flow equivalence are known. 
One is due to M. Boyle and D. Handelman \cite{BH96Israel} and 
the other is due to R. Bowen, J. Franks, W. Parry and D. Sullivan 
\cite{PS75Top,BF77Ann,F84ETDS} 
(see Theorem \ref{flow=barH} and Theorem \ref{flow=BF}). 
By using the former characterization and the groupoid approach, we show that 
if $(X_A,\sigma_A)$ and $(X_B,\sigma_B)$ are continuously orbit equivalent, 
then 
$(\bar X_A,\bar\sigma_A)$ and $(\bar X_B,\bar\sigma_B)$ are flow equivalent 
(Theorem \ref{coe>fe}). 
This, together with the second characterization, implies 
$\det(\id-A)=\det(\id-B)$, and so the conjecture is confirmed. 
It is known that flow equivalence has a close relationship to 
stable isomorphism of Cuntz-Krieger algebras 
(\cite{C81Invent,CK80Invent,F84ETDS,H94ETDS,H95crelle,R95K}). 
As a corollary of the main result, we also prove that 
two-sided irreducible topological Markov shifts 
$(\bar X_A,\bar\sigma_A)$ and $(\bar X_B,\bar\sigma_B)$ are flow equivalent 
if and only if there exists an isomorphism 
between the stable Cuntz-Krieger algebras 
$\mathcal{O}_A\otimes\K$ and $\mathcal{O}_B\otimes\K$ 
preserving their canonical maximal abelian subalgebras 
(Corollary \ref{fe=cartaniso}).

\section{Preliminaries}

We write $\Z_+=\N\cup\{0\}$. 
The transpose of a matrix $A$ is written $A^t$. 
The characteristic function of a set $S$ is denoted by $1_S$. 
We say that a subset of a topological space is clopen 
if it is both closed and open. 
A topological space is said to be totally disconnected 
if its topology is generated by clopen subsets. 
By a Cantor set, 
we mean a compact, metrizable, totally disconnected space 
with no isolated points. 
It is known that any two such spaces are homeomorphic. 
A good introduction to symbolic dynamics can be found 
in the standard textbook \cite{LM} by D. Lind and B. Marcus. 

Let $A=[A(i,j)]_{i,j=1}^N$ be an $N\times N$ matrix with entries in $\{0,1\}$, 
where $1<N\in\N$. 
Throughout the paper, 
we assume that $A$ has no rows or columns identically equal to zero. 
Define 
\[
X_A=\left\{(x_n)_{n\in\N}\in\{1,\dots,N\}^\N\mid 
A(x_n,x_{n+1})=1\quad\forall n\in\N\right\}. 
\]
It is a compact Hausdorff space 
with natural product topology on $\{1,\dots,N\}^\N$. 
The shift transformation $\sigma_A$ on $X_A$ 
defined by $\sigma_A((x_n)_n)=(x_{n+1})_n$ is 
a continuous surjective map on $X_A$. 
The topological dynamical system $(X_A,\sigma_A)$ is called 
the (right) one-sided topological Markov shift for $A$. 
We henceforth assume that $A$ satisfies condition (I) 
in the sense of \cite{CK80Invent}. 
The matrix $A$ satisfies condition (I) if and only if 
$X_A$ has no isolated points, i.e. $X_A$ is a Cantor set. 

We let $(\bar X_A,\bar\sigma_A)$ denote 
the two-sided topological Markov shift. 
Namely, 
\[
\bar X_A=\left\{(x_n)_{n\in\Z}\in\{1,\dots,N\}^\Z\mid 
A(x_n,x_{n+1})=1\quad\forall n\in\Z\right\}
\]
and $\bar\sigma_A((x_n)_{n\in\Z})=(x_{n+1})_{n\in\Z}$. 

A subset $S$ in $X_A$ (resp. in $\bar X_A$) is said to be 
$\sigma_A$-invariant (resp. $\bar\sigma_A$-invariant) 
if $\sigma_A(S)=S$ (resp. $\bar\sigma_A(S)=S$).

\subsection{Continuous orbit equivalence}

For $x=(x_n)_{n\in\N}\in X_A$, 
the orbit $\orb_{\sigma_A}(x)$ of $x$ under $\sigma_A$ is defined by 
\[
\orb_{\sigma_A}(x)=
\bigcup_{k=0}^\infty\bigcup_{l=0}^\infty\sigma_A^{-k}(\sigma_A^l(x)). 
\]

\begin{definition}[{\cite{Matsumoto10PJM}}]\label{defofcoe}
Let $(X_A,\sigma_A)$ and $(X_B,\sigma_B)$ be 
two one-sided topological Markov shifts. 
If there exists a homeomorphism $h:X_A\to X_B$ 
such that $h(\orb_{\sigma_A}(x))=\orb_{\sigma_B}(h(x))$ for $x\in X_A$, 
then $(X_A,\sigma_A)$ and $(X_B,\sigma_B)$ are said to be 
topologically orbit equivalent. 
In this case, there exist $k_1,l_1:X_A\to\Z_+$ such that 
\[
\sigma_B^{k_1(x)}(h(\sigma_A(x)))=\sigma_B^{l_1(x)}(h(x))
\quad\forall x\in X_A. 
\]
Similarly there exist $k_2,l_2:X_B\to\Z_+$ such that 
\[
\sigma_A^{k_2(x)}(h^{-1}(\sigma_B(x)))=\sigma_A^{l_2(x)}(h^{-1}(x))
\quad\forall x\in X_B. 
\]
Furthermore, 
if we may choose $k_1,l_1:X_A\to\Z_+$ and $k_2,l_2:X_B\to\Z_+$ 
as continuous maps, 
the topological Markov shifts $(X_A,\sigma_A)$ and $(X_B,\sigma_B)$ 
are said to be continuously orbit equivalent. 
\end{definition}

If two one-sided topological Markov shifts are topologically conjugate, 
then they are continuously orbit equivalent. 
For the two matrices 
\[
A=\begin{bmatrix}1&1\\1&1\end{bmatrix}\quad\text{and}\quad 
B=\begin{bmatrix}1&1\\1&0\end{bmatrix}, 
\]
the topological Markov shifts $(X_A,\sigma_A)$ and $(X_B,\sigma_B)$ are 
continuously orbit equivalent, 
but not topologically conjugate (see \cite[Section 5]{Matsumoto10PJM}). 

Let $[\sigma_A]$ denote the set of all homeomorphism $\tau$ of $X_A$ 
such that $\tau(x)\in\orb_{\sigma_A}(x)$ for all $x\in X_A$. 
It is called the full group of $(X_A,\sigma_A)$. 
Let $\Gamma_A$ be the set of all $\tau$ in $[\sigma_A]$ 
such that there exist continuous functions $k,l:X_A \to\Z_+$ satisfying 
$\sigma_A^{k(x)}(\tau(x))=\sigma_A^{l(x)}(x)$ for all $x\in X_A$. 
The set $\Gamma_A$ is a subgroup of $[\sigma_A]$ and 
is called the continuous full group for $(X_A,\sigma_A)$. 
We note that the group $\Gamma_A$ has been written as $[\sigma_A]_c$ 
in the earlier paper \cite{Matsumoto10PJM}. 
It has been proved in \cite{Matsumoto12} that 
the isomorphism class of $\Gamma_A$ as an abstract group is 
a complete invariant of the continuous orbit equivalence class 
of $(X_A,\sigma_A)$ 
(see \cite{Matui12} for more general results and further studies).

\subsection{\'Etale groupoids}

By an \'etale groupoid 
we mean a second countable locally compact Hausdorff groupoid 
such that the range map is a local homeomorphism. 
We refer the reader to \cite{R08Irish} 
for background material on \'etale groupoids. 
For an \'etale groupoid $G$, 
we let $G^{(0)}$ denote the unit space and 
let $s$ and $r$ denote the source and range maps. 
For $x\in G^{(0)}$, 
$r(Gx)$ is called the $G$-orbit of $x$. 
When every $G$-orbit is dense in $G^{(0)}$, $G$ is said to be minimal. 
For $x\in G^{(0)}$, 
we write $G_x=r^{-1}(x)\cap s^{-1}(x)$ and call it the isotropy group of $x$. 
The isotropy bundle is 
$G'=\{g\in G\mid r(g)=s(g)\}=\bigcup_{x\in G^{(0)}}G_x$. 
We say that $G$ is principal if $G'=G^{(0)}$. 
When the interior of $G'$ is $G^{(0)}$, 
we say that $G$ is essentially principal. 
A subset $U\subset G$ is called a $G$-set if $r|U,s|U$ are injective. 
For an open $G$-set $U$, 
we let $\pi_U$ denote the homeomorphism $r\circ(s|U)^{-1}$ 
from $s(U)$ to $r(U)$. 

We would like to recall 
the notion of topological full groups for \'etale groupoids. 

\begin{definition}[{\cite[Definition 2.3]{Matui12PLMS}}]
Let $G$ be an essentially principal \'etale groupoid 
whose unit space $G^{(0)}$ is compact. 
\begin{enumerate}
\item The set of all $\alpha\in\Homeo(G^{(0)})$ such that 
for every $x\in G^{(0)}$ there exists $g\in G$ 
satisfying $r(g)=x$ and $s(g)=\alpha(x)$ 
is called the full group of $G$ and denoted by $[G]$. 
\item The set of all $\alpha\in\Homeo(G^{(0)})$ for which 
there exists a compact open $G$-set $U$ satisfying $\alpha=\pi_U$ 
is called the topological full group of $G$ and denoted by $[[G]]$. 
\end{enumerate}
Obviously $[G]$ is a subgroup of $\Homeo(G^{(0)})$ and 
$[[G]]$ is a subgroup of $[G]$. 
\end{definition}

For $\alpha\in[[G]]$ the compact open $G$-set $U$ as above uniquely exists, 
because $G$ is essentially principal. 
Since $G$ is second countable, it has countably many compact open subsets, 
and so $[[G]]$ is at most countable. 
For minimal groupoids on Cantor sets, 
it is known that the isomorphism class of $[[G]]$ is 
a complete invariant of $G$ (\cite[Theorem 3.10]{Matui12}). 

Let $(X_A,\sigma_A)$ be a topological Markov shift. 
The \'etale groupoid $G_A$ for $(X_A,\sigma_A)$ is given by 
\[
G_A=\left\{(x,n,y)\in X_A\times\Z\times X_A\mid
\exists k,l\in\Z_+,\ n=k{-}l,\ \sigma_A^k(x)=\sigma_A^l(y)\right\}. 
\]
The topology of $G_A$ is generated by the sets 
\[
\left\{(x,k{-}l,y)\in G_A
\mid x\in V,\ y\in W,\ \sigma_A^k(x)=\sigma_B^l(y)\right\}, 
\]
where $V,W\subset X_A$ are open and $k,l\in\Z_+$. 
Two elements $(x,n,y)$ and $(x',n',y')$ in $G_A$ are composable 
if and only if $y=x'$, and the multiplication and the inverse are 
\[
(x,n,y)\cdot(y,n',y')=(x,n{+}n',y'),\quad (x,n,y)^{-1}=(y,-n,x). 
\]
The range and source maps are given 
by $r(x,n,y)=(x,0,x)$ and $s(x,n,y)=(y,0,y)$. 
We identify $X_A$ with the unit space $G_A^{(0)}$ via $x\mapsto(x,0,x)$. 
The groupoid $G_A$ is essentially principal. 
The groupoid $G_A$ is minimal if and only if 
$(X_A,\sigma_A)$ is irreducible. 
It is easy to see that 
the topological full group $[[G_A]]$ is canonically isomorphic to 
the continuous full group $\Gamma_A$.

\subsection{Cuntz-Krieger algebras}

Let $A=[A(i,j)]_{i,j=1}^N$ be an $N\times N$ matrix with entries in $\{0,1\}$ 
and let $(X_A,\sigma_A)$ be the one-sided topological Markov shift. 
The Cuntz-Krieger algebra $\mathcal{O}_A$, introduced in \cite{CK80Invent}, 
is the universal $C^*$-algebra 
generated by $N$ partial isometries $S_1,\dots,S_N$ subject to the relations 
\[
\sum_{j=1}^NS_jS_j^*=1\quad\text{and}\quad 
S_i^*S_i=\sum_{j=1}^NA(i,j)S_jS_j^*. 
\]
The subalgebra $\mathcal{D}_A$ of $\mathcal{O}_A$ 
generated by elements $S_{i_1}S_{i_2}\dots S_{i_k}S_{i_k}^*\dots S_{i_1}^*$ 
is naturally isomorphic to $C(X_A)$, 
and is a Cartan subalgebra in the sense of \cite{R08Irish}. 
It is also well-known that 
the pair $(\mathcal{O}_A,\mathcal{D}_A)$ is isomorphic to 
the pair $(C^*_r(G_A),C(X_A))$, 
where $C^*_r(G_A)$ denotes the reduced groupoid $C^*$-algebra and 
$C(X_A)$ is regarded as a subalgebra of it. 
Thus, 
there exists an isomorphism $\Psi:\mathcal{O}_A\to C^*_r(G)$ 
such that $\Psi(\mathcal{D}_A)=C(X_A)$. 

\begin{theorem}\label{coe=cartaniso}
Let $(X_A,\sigma_A)$ and $(X_B,\sigma_B)$ be 
two irreducible one-sided topological Markov shifts. 
The following conditions are equivalent. 
\begin{enumerate}
\item $(X_A,\sigma_A)$ and $(X_B,\sigma_B)$ are 
continuously orbit equivalent. 
\item The \'etale groupoids $G_A$ and $G_B$ are isomorphic. 
\item There exists an isomorphism $\Psi:\mathcal{O}_A\to\mathcal{O}_B$ 
such that $\Psi(\mathcal{D}_A)=\mathcal{D}_B$. 
\end{enumerate}
\end{theorem}
\begin{proof}
The equivalence between (1) and (3) 
follows from \cite[Theorem 1.1]{Matsumoto10PJM}. 
The equivalence between (2) and (3) 
follows from \cite[Proposition 4.11]{R08Irish} 
(see also \cite[Theorem 5.1]{Matui12PLMS}). 
\end{proof}

\subsection{Flow equivalence}

In this subsection, 
we would like to recall Boyle-Handelman's theorem, 
which says that the ordered cohomology group is a complete invariant 
for flow equivalence between irreducible shifts of finite type. 

Let $A=[A(i,j)]_{i,j=1}^N$ be an $N\times N$ matrix with entries in $\{0,1\}$ 
and consider the two-sided topological Markov shift $(\bar X_A,\bar\sigma_A)$. 
Set 
\[
\bar H^A
=C(\bar X_A,\Z)/\{\xi-\xi\circ\bar\sigma_A\mid\xi\in C(\bar X_A,\Z)\}. 
\]
The equivalence class of a function $\xi\in C(\bar X_A,\Z)$ in $\bar H^A$ 
is written $[\xi]$. 
We define the positive cone $\bar H^A_+$ by 
\[
\bar H^A_+=\{[\xi]\in\bar H^A
\mid\xi(x)\geq0\quad\forall x\in\bar X_A\}. 
\]
The pair $(\bar H^A,\bar H^A_+)$ is called 
the ordered cohomology group of $(\bar X_A,\bar\sigma_A)$ 
(see \cite[Section 1.3]{BH96Israel}). 
The following theorem was proved by M. Boyle and D. Handelman, 
which plays a key role in this paper. 

\begin{theorem}[{\cite[Theorem 1.12]{BH96Israel}}]\label{flow=barH}
Suppose that $(\bar X_A,\bar\sigma_A)$ and $(\bar X_B,\bar\sigma_B)$ are 
irreducible two-sided topological Markov shifts. 
Then the following are equivalent. 
\begin{enumerate}
\item $(\bar X_A,\bar\sigma_A)$ and $(\bar X_B,\bar\sigma_B)$ are 
flow equivalent. 
\item The ordered cohomology groups 
$(\bar H^A,\bar H^A_+)$ and $(\bar H^B,\bar H^B_+)$ are isomorphic, 
i.e. there exists an isomorphism 
$\Phi:\bar H^A\to\bar H^B$ such that $\Phi(\bar H^A_+)=\bar H^B_+$. 
\end{enumerate}
\end{theorem}

We also recall the following from \cite{BH96Israel} for later use. 

\begin{proposition}[{\cite[Proposition 3.13 (a)]{BH96Israel}}]
\label{whenpositive}
Let $(\bar X_A,\bar\sigma_A)$ be a two-sided topological Markov shift 
and let $\xi\in C(\bar X_A,\Z)$. 
Then $[\xi]$ is in $\bar H^A_+$ if and only if 
\[
\sum_{x\in O}\xi(x)\geq0
\]
holds for any finite $\bar\sigma_A$-invariant set $O\subset\bar X_A$. 
\end{proposition}

In the same way as above, we introduce $(H^A,H^A_+)$ 
for the one-sided topological Markov shift $(X_A,\sigma_A)$ as follows: 
\[
H^A=C(X_A,\Z)/\{\xi-\xi\circ\sigma_A\mid\xi\in C(X_A,\Z)\}
\]
and 
\[
H^A_+=\{[\xi]\in H^A\mid\xi(x)\geq0\quad\forall x\in X_A\}. 
\]
We will show that 
$(\bar H^A,\bar H^A_+)$ and $(H^A,H^A_+)$ are actually isomorphic 
(Lemma \ref{H=barH}).

\subsection{The Bowen-Franks group}

Let $A=[A(i,j)]_{i,j=1}^N$ be an $N\times N$ matrix with entries in $\{0,1\}$. 
The Bowen-Franks group $\BF(A)$ is the abelian group $\Z^N/(\id-A)\Z^N$. 
R. Bowen and J. Franks \cite{BF77Ann} have proved that 
the Bowen-Franks group is an invariant of flow equivalence. 
W. Parry and D. Sullivan \cite{PS75Top} have proved that 
the determinant of $\id-A$ is also an invariant of flow equivalence. 
Evidently, if $\BF(A)$ is an infinite group, then $\det(\id-A)$ is zero. 
If $\BF(A)$ is a finite group, 
then $\lvert\det(\id-A)\rvert$ is equal to the cardinality of $\BF(A)$. 
Therefore it is sufficient to know the Bowen-Franks group and 
the sign of the determinant in order to find the determinant. 
The following theorem by Franks shows that these invariants are complete. 

\begin{theorem}[{\cite{F84ETDS}}]\label{flow=BF}
Suppose that $(\bar X_A,\bar\sigma_A)$ and $(\bar X_B,\bar\sigma_B)$ are 
irreducible two-sided topological Markov shifts. 
Then 
$(\bar X_A,\bar\sigma_A)$ and $(\bar X_B,\bar\sigma_B)$ are flow equivalent 
if and only if 
$\BF(A)\cong\BF(B)$ and $\det(\id-A)=\det(\id-B)$. 
\end{theorem}

In what follows we consider $\BF(A^t)=\Z^N/(\id-A^t)\Z^N$. 
Although $\BF(A^t)$ is isomorphic to $\BF(A)$ as an abelian group, 
there does not exist a canonical isomorphism between them, 
and so we must distinguish them carefully. 

We denote the equivalence class of $(1,1,\dots,1)\in\Z^N$ in $\BF(A^t)$ 
by $u_A$. 
By \cite[Proposition 3.1]{C81Invent}, 
$K_0(\mathcal{O}_A)$ is isomorphic to $\BF(A^t)$ and 
the class of the unit of $\mathcal{O}_A$ maps to $u_A$ under this isomorphism. 
And $K_1(\mathcal{O}_A)$ is isomorphic to $\Ker(\id-A^t)$ on $\Z^N$. 
In \cite{Matui12PLMS}, 
it has been shown that these groups naturally arise 
from the homology theory of \'etale groupoids. 

Let $G$ be an \'etale groupoid whose unit space $G^{(0)}$ is a Cantor set. 
One can associate the homology groups $H_n(G)$ with $G$ 
(see \cite[Section 3]{Matui12PLMS} for the precise definition). 
The homology group $H_0(G)$ is the quotient of $C(G^{(0)},\Z)$ by the subgroup 
generated by $1_{r(U)}-1_{s(U)}$ for compact open $G$-sets $U$. 
We denote the equivalence class of $\xi\in C(G^{(0)},\Z)$ in $H_0(G)$ 
by $[\xi]$. 
For the \'etale groupoid $G_A$, we have the following. 

\begin{theorem}[{\cite[Theorem 4.14]{Matui12PLMS}}]\label{HofGA}
Let $(X_A,\sigma_A)$ be a one-sided topological Markov shift. 
Then 
\[
H_n(G_A)\cong\begin{cases}\BF(A^t)=\Z^N/(\id-A^t)\Z^N&n=0\\
\Ker(\id-A^t)&n=1\\0&n\geq2. \end{cases}
\]
Moreover, there exists an isomorphism $\Phi:H_0(G_A)\to\BF(A^t)$ 
such that $\Phi([1_{X_A}])=u_A$. 
\end{theorem}

In particular, it follows from Theorem \ref{coe=cartaniso} that 
the pair $(\BF(A^t),u_A)$ is an invariant 
for continuous orbit equivalence of one-sided topological Markov shifts 
(see also \cite[Theorem 1.3]{Matsumoto13DCDS}). 
Thus, 
if $(X_A,\sigma_A)$ and $(X_B,\sigma_B)$ are continuously orbit equivalent, 
then there exists an isomorphism $\Phi:\BF(A^t)\to\BF(B^t)$ 
such that $\Phi(u_A)=u_B$.

\section{Classification up to continuous orbit equivalence}

Let $(X_A,\sigma_A)$ be an irreducible one-sided topological Markov shift. 
As in the previous section, $(\bar X_A,\bar\sigma_A)$ denotes 
the two-sided topological Markov shift corresponding to $(X_A,\sigma_A)$. 
Define $\rho:\bar X_A\to X_A$ by $\rho((x_n)_{n\in\Z})=(x_n)_{n\in\N}$. 
Clearly we have $\sigma_A\circ\rho=\rho\circ\bar\sigma_A$. 

\begin{lemma}\label{H=barH}
The map $C(X_A,\Z)\ni\xi\mapsto\xi\circ\rho\in C(\bar X_A,\Z)$ gives rise to 
an isomorphism $\tilde\rho$ from $H^A$ to $\bar H^A$ 
satisfying $\tilde\rho(H^A_+)=\bar H^A_+$. 
\end{lemma}
\begin{proof}
For any $\eta\in C(X_A,\Z)$, one has 
$(\eta-\eta\circ\sigma_A)\circ\rho
=\eta\circ\rho-\eta\circ\rho\circ\bar\sigma_A$, 
and so $[\xi]\mapsto[\xi\circ\rho]$ is a well-defined homomorphism 
$\tilde\rho$ from $H^A$ to $\bar H^A$. 

Let $\zeta\in C(\bar X_A,\Z)$. 
Then $\zeta(x)$ depends only on finitely many coordinates of $x\in\bar X_A$. 
Hence, for sufficiently large $n\in\N$, 
there exists $\xi\in C(X_A,\Z)$ 
such that $\zeta\circ\bar\sigma_A^n=\xi\circ\rho$. 
Thus $\tilde\rho$ is surjective. 

Clearly $\tilde\rho(H^A_+)\subset\bar H^A_+$. 
It follows from the argument above that 
$\bar H^A_+$ is contained in $\tilde\rho(H^A_+)$. 

It remains for us to show the injectivity. 
Let $\xi\in C(X_A,\Z)$. 
Suppose that there exists $\zeta\in C(\bar X_A,\Z)$ 
such that $\xi\circ\rho=\zeta-\zeta\circ\bar\sigma_A$. 
In the same way as above, for sufficiently large $n\in\N$, 
there exists $\eta\in C(X_A,\Z)$ 
such that $\zeta\circ\bar\sigma_A^n=\eta\circ\rho$. 
Then 
\[
\xi\circ\sigma_A^n\circ\rho=\xi\circ\rho\circ\bar\sigma_A^n
=\zeta\circ\bar\sigma_A^n-\zeta\circ\bar\sigma_A^{n+1}
=(\eta-\eta\circ\sigma_A)\circ\rho. 
\]
Hence $\xi\circ\sigma_A^n=\eta-\eta\circ\sigma_A$. 
Thus $[\xi]=[\xi\circ\sigma_A^n]=0$ in $H^A$. 
\end{proof}

\begin{lemma}\label{whenpositive2}
For $\xi\in C(X_A,\Z)$, 
$[\xi]$ is in $H^A_+$ if and only if 
$\sum_{x\in O}\xi(x)\geq0$ holds 
for every finite $\sigma_A$-invariant set $O\subset X_A$. 
\end{lemma}
\begin{proof}
Suppose that $[\xi]$ is in $H^A_+$. 
By the lemma above, $\tilde\rho([\xi])=[\xi\circ\rho]$ is in $\bar H^A_+$. 
Let $O\subset X_A$ be a finite $\sigma_A$-invariant set. 
There exists a finite $\bar\sigma_A$-invariant set $\bar O\subset\bar X_A$ 
such that $\rho|\bar O$ is a bijection from $\bar O$ to $O$. 
It follows from Proposition \ref{whenpositive} that 
$\sum_{x\in\bar O}\xi(\rho(x))\geq0$. 
Hence $\sum_{x\in O}\xi(x)\geq0$. 

Suppose that $\sum_{x\in O}\xi(x)\geq0$ holds 
for every finite $\sigma_A$-invariant set $O\subset X_A$. 
For any finite $\bar\sigma_A$-invariant set $\bar O\subset\bar X_A$, 
$O=\rho(\bar O)\subset X_A$ is a finite $\sigma_A$-invariant set and 
$\rho|\bar O$ is injective. 
Therefore $\sum_{x\in\bar O}\xi(\rho(x))=\sum_{x\in O}\xi(x)\geq0$. 
By Proposition \ref{whenpositive}, $[\xi\circ\rho]$ is in $\bar H^A_+$. 
By the lemma above, $[\xi]$ is in $H^A_+$ as desired. 
\end{proof}

Let $G$ be an \'etale groupoid. 
We denote by $\Hom(G,\Z)$ the set of continuous homomorphisms $\omega:G\to\Z$. 
We think of $\Hom(G,\Z)$ as an abelian group by pointwise addition. 
For $\xi\in C(G^{(0)},\Z)$, we can define $\partial(\xi)\in\Hom(G,\Z)$ 
by $\partial(\xi)(g)=\xi(r(g))-\xi(s(g))$. 
The cohomology group $H^1(G)=H^1(G,\Z)$ is the quotient of $\Hom(G,\Z)$ 
by $\{\partial(\xi)\mid\xi\in C(G^{(0)},\Z)\}$. 
The equivalence class of $\omega:G\to\Z$ is written $[\omega]\in H^1(G)$. 

Let $g\in G$ be such that $r(g)=s(g)$, that is, $g\in G'$. 
Since $\partial(\xi)(g)=0$ for any $\xi\in C(G^{(0)},\Z)$, 
$[\omega]\mapsto\omega(g)$ is a well-defined homomorphism 
from $H^1(G)$ to $\Z$. 
We say that $g$ is attracting if there exists a compact open $G$-set $U$ 
such that $g\in U$, $r(U)\subset s(U)$ and 
\[
\lim_{n\to+\infty}(\pi_U)^n(y)=r(g)
\]
holds for any $y\in s(U)$. 

Let $(X_A,\sigma_A)$ be a one-sided topological Markov shift and 
consider the \'etale groupoid $G_A$ (see Section 2.2 for the definition). 
We say that $x\in X_A$ is eventually periodic 
if there exist $k,l\in\Z_+$ 
such that $k\neq l$ and $\sigma_A^k(x)=\sigma_A^l(x)$. 
This is equivalent to saying that 
$\{\sigma_A^n(x)\in X_A\mid n\in\Z_+\}$ is a finite set. 
When $x$ is eventually periodic, we call 
\[
\min\left\{k-l\mid k,l\in\Z_+,\ k>l,\ \sigma_A^k(x)=\sigma_A^l(x)\right\}
\]
the period of $x$. 

\begin{lemma}
Let $x\in X_A$. 
\begin{enumerate}
\item If $x$ is not eventually periodic, then 
the isotropy group $(G_A)_x$ is trivial. 
\item If $x$ is eventually periodic, 
then $(G_A)_x=\{(x,np,x)\in G_A\mid n\in\Z\}\cong\Z$, 
where $p$ is the period of $x$. 
\item When $x$ is eventually periodic and has period $p$, 
$(x,np,x)$ is attracting if and only if $n$ is positive. 
\end{enumerate}
\end{lemma}
\begin{proof}
(1) and (2) are obvious. 
We prove (3). 
Suppose that $x$ is an eventually periodic point whose period is $p$. 
Let $(x,np,x)\in (G_A)_x$. 
Assume that $n$ is positive. 
Choose $k,l\in\Z_+$ so that $\sigma_A^k(x)=\sigma_A^l(x)$ and $pn=k-l$. 
Define a clopen neighborhood $V$ and $W$ of $x$ by 
\[
V=\{(y_n)_n\in X_A\mid y_i=x_i\quad\forall i=1,2,\dots,k{+}1\}
\]
and 
\[
W=\{(y_n)_n\in X_A\mid y_i=x_i\quad\forall i=1,2,\dots,l{+}1\}. 
\]
We have $V\subset W$ and $\sigma_A^k(V)=\sigma_A^l(W)$. 
Then 
\[
U=\left\{(y,np,z)\in G_A
\mid y\in V,\ z\in W,\ \sigma_A^k(y)=\sigma_A^l(z)\right\}
\]
is a compact open $G_A$-set 
such that $(x,np,x)\in U$, $r(U)=V$, $s(U)=W$ and 
$\pi_U=(\sigma_A^k|V)^{-1}\circ(\sigma_A^l|W)$. 
It is easy to see that 
\[
\lim_{m\to+\infty}(\pi_U)^m(z)=x
\]
holds for any $z\in s(U)$. 
Thus $(x,np,x)$ is attracting. 

Suppose that 
$U\subset G_A$ is a compact open $G_A$-set containing $(x,0,x)$. 
Then $\pi_U(y)=y$ for any $y$ sufficiently close to $x$, 
and so $(x,0,x)$ is not attracting. 

Assume that $n$ is negative. 
Let $U\subset G_A$ be a compact open $G_A$-set containing $(x,np,x)$. 
By the argument above, $(x,-np,x)$ is attracting. 
Hence there exists a clopen neighborhood $V$ of $x$ 
such that $V\subset s(U)$ and $V\subset\pi_U(V)$. 
This means that $(x,np,x)$ cannot be an attracting element. 
\end{proof}

\begin{proposition}
There exists an isomorphism $\Phi:H^1(G_A)\to H^A$ such that 
$\Phi([\omega])$ is in $H^A_+$ if and only if 
$\omega(g)\geq0$ for every attracting $g\in G_A$. 
\end{proposition}
\begin{proof}
Let $\omega\in\Hom(G_A,\Z)$. 
Define $\xi\in C(X_A,\Z)$ by 
\[
\xi(x)=\omega((x,1,\sigma_A(x))). 
\]
Let us verify that the map $\omega\mapsto\xi$ is surjective. 
For a given $\xi\in C(X_A,\Z)$, we can define $\omega\in\Hom(G_A,\Z)$ 
as follows. 
Take $(x,n,y)\in G_A$. 
There exists $k,l\in\Z_+$ 
such that $k-l=n$ and $\sigma_A^k(x)=\sigma_A^l(y)$. 
Put 
\[
\omega((x,n,y))
=\sum_{i=0}^{k-1}\xi(\sigma_A^i(x))-\sum_{j=0}^{l-1}\xi(\sigma_A^j(y)). 
\]
Clearly this gives a well-defined continuous homomorphism 
from $G_A$ to $\Z$. 
If there exists $\eta\in C(X_A,\Z)$ such that $\omega=\partial(\eta)$, 
then $\xi=\eta-\eta\circ\sigma_A$, that is, $[\xi]=0$ in $H^A$. 
It is also easy to see that the converse holds. 
Therefore $\Phi:[\omega]\mapsto[\xi]$ is an isomorphism 
from $H^1(G_A)$ to $H^A$. 

We would like to show that $[\xi]$ is in $H^A_+$ 
if and only if $\omega(g)\geq0$ for every attracting $g\in G_A$. 
Let $x\in X_A$ be an eventually periodic point whose period is $p$ 
and let $g=(x,np,x)$ be an attracting element. 
By the lemma above, $n$ is positive. 
There exists $k,l\in\Z_+$ 
such that $k-l=np$ and $\sigma_A^k(x)=\sigma_A^l(x)$. 
Then one has 
\begin{align*}
\omega(g)
&=\sum_{i=0}^{k-1}\xi(\sigma_A^i(x))-\sum_{j=0}^{l-1}\xi(\sigma_A^j(x))\\
&=\sum_{i=l}^{k-1}\xi(\sigma_A^i(x))\\
&=n\sum_{i=l}^{l+p-1}\xi(\sigma_A^i(x)). 
\end{align*}
Notice that $O=\{\sigma_A^l(x),\sigma_A^{l+1}(x),\dots,\sigma_A^{l+p-1}(x)\}$ 
is a finite $\sigma_A$-invariant set. 
By Lemma \ref{whenpositive2}, 
$[\xi]$ belongs to $H^A_+$ if and only if 
\[
\sum_{y\in O}\xi(y)\geq0
\]
for any finite $\sigma_A$-invariant set $O\subset X_A$, 
thereby completing the proof. 
\end{proof}

Consequently we have the following. 

\begin{theorem}\label{coe>fe}
Let $(X_A,\sigma_A)$ and $(X_B,\sigma_B)$ be 
two irreducible one-sided topological Markov shifts. 
If $(X_A,\sigma_A)$ is continuously orbit equivalent to $(X_B,\sigma_B)$, 
then there exists an isomorphism $\Phi:H^A\to H^B$ 
such that $\Phi(H^A_+)=H^B_+$. 
In particular, 
$(\bar X_A,\bar\sigma_A)$ is flow equivalent to $(\bar X_B,\bar\sigma_B)$. 
\end{theorem}
\begin{proof}
Consider the \'etale groupoids $G_A$ and $G_B$. 
By Theorem \ref{coe=cartaniso}, $G_A$ and $G_B$ are isomorphic. 
Let $\phi:G_A\to G_B$ be an isomorphism. 
For $g\in G_A$, 
$g$ is attracting in $G_A$ if and only if $\phi(g)$ is attracting in $G_B$. 
It follows from the proposition above that 
$(H^A,H^A_+)$ is isomorphic to $(H^B,H^B_+)$. 
Then, Lemma \ref{H=barH} implies that 
$(\bar H^A,\bar H^A_+)$ is isomorphic to $(\bar H^B,\bar H^B_+)$. 
By Theorem \ref{flow=barH}, 
$(\bar X_A,\bar\sigma_A)$ is flow equivalent to $(\bar X_B,\bar\sigma_B)$. 
\end{proof}

\begin{theorem}\label{classification}
Let $(X_A,\sigma_A)$ and $(X_B,\sigma_B)$ be 
two irreducible one-sided topological Markov shifts. 
The following conditions are equivalent. 
\begin{enumerate}
\item $(X_A,\sigma_A)$ and $(X_B,\sigma_B)$ are 
continuously orbit equivalent. 
\item The \'etale groupoids $G_A$ and $G_B$ are isomorphic. 
\item There exists an isomorphism $\Psi:\mathcal{O}_A\to\mathcal{O}_B$ 
such that $\Psi(\mathcal{D}_A)=\mathcal{D}_B$. 
\item $\mathcal{O}_A$ is isomorphic to $\mathcal{O}_B$ 
and $\det(\id-A)=\det(\id-B)$. 
\item There exists an isomorphism $\Phi:\BF(A^t)\to\BF(B^t)$ 
such that $\Phi(u_A)=u_B$ and $\det(\id-A)=\det(\id-B)$. 
\end{enumerate}
\end{theorem}
\begin{proof}
The equivalence between (1), (2) and (3) is already known 
(Theorem \ref{coe=cartaniso}). 
As mentioned in Section 2.5, 
$(K_0(\mathcal{O}_A),[1])$ is isomorphic to $(\BF(A^t),u_A)$, 
and so (4)$\Rightarrow$(5) holds. 
The implication (5)$\Rightarrow$(1) follows 
from \cite[Theorem 1.1]{Matsumoto13PAMS}. 

Suppose that $(X_A,\sigma_A)$ and $(X_B,\sigma_B)$ are 
continuously orbit equivalent. 
It follows from the theorem above that 
the two-sided topological Markov shifts 
$(\bar X_A,\bar\sigma_A)$ and $(\bar X_B,\bar\sigma_B)$ are flow equivalent. 
Therefore, by \cite{PS75Top}, we have $\det(\id-A)=\det(\id-B)$ 
(see Theorem \ref{flow=BF}). 
Since (1)$\Rightarrow$(3) is already known, 
$\mathcal{O}_A$ is isomorphic to $\mathcal{O}_B$. 
Thus we have obtained (4). 
This completes the proof. 
\end{proof}

As mentioned in Section 2.5, 
$\det(\id-A)=0$ when $\BF(A^t)$ is infinite, and 
$\lvert\det(\id-A)\rvert$ equals the cardinality of $\BF(A^t)$ 
when $\BF(A^t)$ is finite. 
Hence, our invariant of the continuous orbit equivalence consists of 
a finitely generated abelian group $F$, an element $u\in F$ 
and $s\in\{-1,0,1\}$ such that $F$ is an infinite group if and only if $s=0$. 
Conversely, for any such triplet $(F,u,s)$, 
there exists an irreducible one-sided topological Markov shift 
whose invariant is equal to $(F,u,s)$. 
This is probably known to experts, 
but the authors are not aware of a specific reference 
and thus include a proof for completeness. 

\begin{lemma}\label{range}
Let $F$ be a finitely generated abelian group and let $u\in F$. 
Let $s=0$ when $F$ is infinite and 
let $s$ be either $-1$ or $1$ when $F$ is finite. 
There exists an irreducible one-sided topological Markov shift 
$(X_A,\sigma_A)$ such that $(F,u)\cong (\BF(A^t),u_A)$ and 
the sign of $\det(\id-A)$ equals $s\in\{-1,0,1\}$. 
\end{lemma}
\begin{proof}
Suppose that we are given $(F,u,s)$. 
It suffices to find a square irreducible matrix $A$ with entries in $\Z_+$ 
satisfying the desired properties 
(see \cite[Section 2.3]{LM} or \cite[Remark 2.16]{CK80Invent}). 
Let $A=[A(i,j)]_{i,j=1}^N$ be an $N\times N$ matrix with entries in $\Z_+$ 
such that $A(1,1)=2$, $A(i,i)\geq2$ and $A(i,j)=1$ 
for all $i,j$ with $i\neq j$. 
Let $d_i=A(i,i)-2$ and let $r=\lvert\{i\mid d_i{=}0\}\rvert-1$. 
Then it is straightforward to see 
\[
\BF(A^t)\cong\Z^r\oplus\bigoplus_{d_i\geq2}\Z/d_i\Z
\quad\text{and}\quad 
\det(\id-A)=(-1)^N\prod_{i=2}^Nd_i. 
\]
Therefore we can construct such $A$ 
so that $\BF(A^t)\cong F$ and the sign of $\det(\id-A)$ equals $s$. 
In what follows we identify $\BF(A^t)$ with $F$. 
Note that $u_A\in\BF(A^t)$ is zero. 
Choose $(c_1,c_2,\dots,c_N)\in\Z^N$ 
whose equivalence class in $\BF(A^t)$ equals $u$. 
Since $u_A$ is zero, we may assume $c_i\in\Z_+$ for all $i$. 
We now construct a new matrix $B$ as follows. 
Set 
\[
\Sigma=\{(i,j)\in\Z_+\times\Z_+\mid1\leq i\leq N,\ 0\leq j\leq c_i\}. 
\]
Define $B=[B((i,j),(k,l))]_{(i,j),(k,l)\in\Sigma}$ by 
\[
B((i,j),(k,l))=\begin{cases}A(i,k)&j=c_i,\ l=0\\
1&i=k,\ j{+}1=l\\0&\text{otherwise. }\end{cases}
\]
The group $\BF(A^t)$ is the abelian group 
with generators $e_1,\dots,e_N$ and relations 
\[
e_i=\sum_{j=1}^NA(i,j)e_j, 
\]
and $u$ equals $\sum c_ie_i$. 
The group $\BF(B^t)$ is the abelian group 
with generators $\{f_{i,j}\mid(i,j)\in\Sigma\}$ and relations 
\[
f_{i,j}=f_{i,j'}\quad\text{and}\quad f_{i,c_i}=\sum_{k=1}^NA(i,k)f_{k,0}, 
\]
and $u_B$ equals $\sum f_{i,j}$. 
Hence $(\BF(A^t),u)$ is isomorphic to $(\BF(B^t),u_B)$. 
It is also easy to see $\det(\id-A)=\det(\id-B)$. 
The proof is completed. 
\end{proof}

For $i=1,2$, let $G_i$ be a minimal essentially principal \'etale groupoid 
whose unit space is a Cantor set. 
It has been shown that the following conditions are mutually equivalent 
(\cite[Theorem 3.10]{Matui12}). 
For a group $\Gamma$, we let $D(\Gamma)$ denote the commutator subgroup. 
\begin{itemize}
\item $G_1$ and $G_2$ are isomorphic as \'etale groupoids. 
\item $[[G_1]]$ and $[[G_2]]$ are isomorphic as discrete groups. 
\item $D([[G_1]])$ and $D([[G_2]])$ are isomorphic as discrete groups. 
\end{itemize}
The \'etale groupoid $G_A$ arising from $(X_A,\sigma_A)$ 
is minimal, essentially principal and purely infinite 
(\cite[Lemma 6.1]{Matui12}). 
Hence $D([[G_A]])$ is simple by \cite[Theorem 4.16]{Matui12}. 
Moreover, $D([[G_A]])$ is finitely generated (\cite[Corollary 6.25]{Matui12}), 
$[[G_A]]$ is of type F$_\infty$ (\cite[Theorem 6.21]{Matui12}) and 
$[[G_A]]/D([[G_A]])$ is isomorphic to $(H_0(G_A)\otimes\Z_2)\oplus H_1(G_A)$ 
(\cite[Corollary 6.24]{Matui12}). 
Theorem \ref{classification} tells us that 
the isomorphism class of $[[G_A]]$ (and $D([[G_A]])$) is determined 
by $(H_0(G_A),[1_{X_A}],\det(\id-A))$ (see also Theorem \ref{HofGA}). 
By Lemma \ref{range}, for each triplet $(F,u,s)$ 
there exists $(X_A,\sigma_A)$ whose invariant agrees with it. 
In particular, the simple finitely generated groups $D([[G_A]])$ are 
parameterized by such triplets $(F,u,s)$. 

We conclude this article by giving a corollary. 
We denote by $\K$ 
the $C^*$-algebra of all compact operators on $\ell^2(\Z)$. 
Let $\mathcal{C}\cong c_0(\Z)$ be the maximal abelian subalgebra of $\K$ 
consisting of diagonal operators. 

\begin{corollary}\label{fe=cartaniso}
Let $(\bar X_A,\bar\sigma_A)$ and $(\bar X_B,\bar\sigma_B)$ be 
two irreducible two-sided topological Markov shifts. 
The following conditions are equivalent. 
\begin{enumerate}
\item $(\bar X_A,\bar\sigma_A)$ and $(\bar X_B,\bar\sigma_B)$ are 
flow equivalent. 
\item There exists an isomorphism 
$\Psi:\mathcal{O}_A\otimes\K\to\mathcal{O}_B\otimes\K$ such that 
$\Psi(\mathcal{D}_A\otimes\mathcal{C})=\mathcal{D}_B\otimes\mathcal{C}$. 
\end{enumerate}
\end{corollary}
\begin{proof}
(1)$\Rightarrow$(2) is known (\cite[Theorem 4.1]{CK80Invent}). 
Let us assume (2). 
In what follows, we identify the Bowen-Franks group 
with the $K_0$-group of the Cuntz-Krieger algebra. 
We have the isomorphism $K_0(\Psi):\BF(A^t)\to\BF(B^t)$. 
By Lemma \ref{range}, there exists 
an irreducible one-sided topological Markov shift $(X_C,\sigma_C)$ 
such that $(\BF(B^t),K_0(\Psi)(u_A))\cong(\BF(C^t),u_C)$ and 
$\det(\id-B)=\det(\id-C)$. 
It follows from Theorem \ref{flow=BF} that 
$(\bar X_B,\bar\sigma_B)$ is flow equivalent to $(\bar X_C,\bar\sigma_C)$. 
Moreover, by Huang's theorem \cite[Theorem 2.15]{H94ETDS} and its proof, 
there exists an isomorphism 
$\Phi:\mathcal{O}_B\otimes\K\to\mathcal{O}_C\otimes\K$ such that 
$\Phi(\mathcal{D}_B\otimes\mathcal{C})=\mathcal{D}_C\otimes\mathcal{C}$ and 
$K_0(\Phi)(K_0(\Psi)(u_A))=u_C$. 
Then $\Phi\circ\Psi$ is an isomorphism 
from $\mathcal{O}_A\otimes\K$ to $\mathcal{O}_C\otimes\K$ such that 
$(\Phi\circ\Psi)(\mathcal{D}_A\otimes\mathcal{C})
=\mathcal{D}_C\otimes\mathcal{C}$ and $K_0(\Phi\circ\Psi)(u_A)=u_C$. 
In the same way as the proof of \cite[Theorem 4.1]{Matsumoto13PAMS}, 
we can conclude that $(\mathcal{O}_A,\mathcal{D}_A)$ is 
isomorphic to $(\mathcal{O}_C,\mathcal{D}_C)$. 
By virtue of Theorem \ref{classification}, 
we get $\det(\id-A)=\det(\id-C)$. 
Therefore $\det(\id-A)=\det(\id-B)$. 
Hence, by Theorem \ref{flow=BF}, 
$(\bar X_A,\bar\sigma_A)$ and $(\bar X_B,\bar\sigma_B)$ are flow equivalent. 
\end{proof}

\end{document}